\documentclass[12pt]{amsart}
\author{Paul Pollack}
\address{University of Georgia\\Department of Mathematics\\Athens, Georgia 30602\\USA}
\email{pollack@uga.edu}


\title[An elemental Erd\H{o}s-Kac theorem]{An elemental Erd\H{o}s--Kac theorem for\\ algebraic number fields}
\usepackage{amsmath,amssymb,amsthm,geometry,url,microtype}
\usepackage[cal=boondoxo]{mathalfa}
\DeclareMathAlphabet{\curly}{U}{rsfs}{m}{n}
\newtheorem{thm}{Theorem}
\newtheorem{prop}[thm]{Proposition}

\newtheorem{lem}[thm]{Lemma}

 \newtheoremstyle{component}{}{}{}{}{\bf}{.}{.5em}{\thmnote{#3}#1}
    \theoremstyle{component}

\theoremstyle{remark}

\newtheorem*{example}{Example}
\newtheorem*{remark}{Remark}
\begin{document}
\renewcommand{\labelenumi}{(\roman{enumi})}
\def\sumprime{\sideset{}{^{'}}{\sum}}
\def\A{\curly{A}}
\def\Cc{\mathcal{C}}
\def\lcm{\mathrm{lcm}}
\def\Q{\mathbf{Q}}
\def\bb{\mathfrak{b}}
\def\Qq{\curly{Q}}
\def\Z{\mathbf{Z}}
\def\R{\curly{R}}
\def\pp{\mathfrak{p}}
\def\ssf{\curly{S}}
\def\Pp{\curly{P}}
\def\aa{\mathfrak{a}}
\def\e{\mathrm{e}}
\def\C{\mathbf{C}}
\def\Oo{\mathcal{O}}
\def\Id{\mathrm{Id}}
\def\Cl{\mathrm{Cl}}
\def\uu{\mathfrak{u}}
\def\dd{\mathfrak{d}}
\def\Prin{\mathrm{Prin}}
\renewcommand{\ssf}{\mathfrak{s}}
\newcommand{\pmid}{\mathrel{\mid^{\ast}}}
\newcommand{\pnmid}{\mathrel{\nmid^{\ast}}}
\begin{abstract} Fix a number field $K$. For each nonzero $\alpha \in \Z_K$, let $\nu(\alpha)$ denote the number of distinct, nonassociate irreducible divisors of $\alpha$. We show that $\nu(\alpha)$ is normally distributed with mean proportional to $(\log\log |N(\alpha)|)^{D}$ and standard deviation proportional to $(\log\log{|N(\alpha)|})^{D-1/2}$. Here $D$, as well as the constants of proportionality, depend only on the class group of $K$. For example, for each fixed real $\lambda$, the proportion of $\alpha \in \Z[\sqrt{-5}]$ with
\[ \nu(\alpha) \le \frac{1}{8}(\log\log{N(\alpha)})^2 + \frac{\lambda}{2\sqrt{2}} (\log\log{N(\alpha)})^{3/2} \]
is given by $\frac{1}{\sqrt{2\pi}} \int_{-\infty}^{\lambda} e^{-t^2/2}\, \mathrm{d}t$. As further evidence that ``irreducibles play a game of chance'', we show that the values $\nu(\alpha)$ are equidistributed modulo $m$ for every fixed $m$.
\end{abstract}
\maketitle

\section{Introduction} The field of probabilistic number theory was born in 1939 out of a fruitful collaboration of Erd\H{o}s and Kac. Let $\omega(n)$ denote the number of distinct prime factors of the positive integer $n$. The celebrated \emph{Erd\H{o}s--Kac theorem} asserts that the quantity
\[\frac{\omega(n)-\log\log{x}}{\sqrt{\log\log{x}}}, \]
thought of as a random variable on the natural numbers $n \le x$ (with the uniform measure), converges in law to a standard Gaussian, as $x\to\infty$ \cite{EK40}. In this statement, $\log\log{x}$ may be changed to $\log\log{n}$ without affecting the meaning, since the two quantities differ by less than $1$ for all $n \in (x^{1/e},x]$. Thus, the theorem is often summarized by saying that $\omega(n)$ is normally distributed with mean $\log\log{n}$ and standard deviation $\sqrt{\log\log{n}}$.

Variants of the Erd\H{o}s--Kac theorem abound (see \cite{dekroon66}, \cite{elliott80}, \cite{liu04}, \cite{KL08}, and the references in \cite{GS07}). In this article, we describe what appears to be a new generalization in the number field setting.

Suppose that $K$ is a number field with ring of integers $\Z_K$. Let $\Id(\Z_K)$ denote the (commutative, cancellative) monoid of nonzero integral ideals of $\Z_K$, and let $\Prin(\Z_K)$ denote the submonoid of principal ideals. For each $\mathfrak{a}\in \Id(\Z_K)$, let $\omega(\mathfrak{a})$ denote the number of distinct prime ideal factors of $\mathfrak{a}$. In \cite{liu04}, Liu proves an $\Id(\Z_K)$-generalization of Erd\H{o}s--Kac, namely that $\omega(\mathfrak{a})$ is normally distributed with mean $\log\log N(\mathfrak{a})$ and standard deviation $\sqrt{\log\log N(\mathfrak{a})}$.

The ``fundamental theorem of ideal theory'' asserts that $\Id(\Z_K)$ is a factorial monoid, with the prime elements in the monoid sense coinciding with the nonzero prime ideals of $\Z_K$. By contrast, $\Prin(\Z_K)$ is in general not factorial, as shown by the famous example
\[ (1+\sqrt{-5}) (1-\sqrt{-5}) = (2)(3) \]
when $K = \Q(\sqrt{-5})$. Notwithstanding the failure of unique factorization, it is still sensible to count the number of irreducible divisors of an element of $\Prin(\Z_K)$ and to ask if something like the Erd\H{o}s--Kac theorem holds. Our main theorem asserts that this is indeed the case.

For each nonzero $\alpha \in \Z_K$, we let $\nu(\alpha)$ denote the number of nonassociate irreducible divisors of $\alpha$. (Equivalently, $\nu(\alpha)$ is the number of irreducible divisors of $(\alpha)$ in the monoid $\Prin(\Z_K)$.) We let $\log_k$ denote  the $k$-fold iterated logarithm.

\begin{thm}\label{thm:main} Let $K$ be a number field. There are positive constants $A$ and $B$, as well as a positive integer $D$, such that the following holds. For each fixed $\lambda > 0$,
\[ \frac{\#\{(\alpha): 0 < |N(\alpha)|\le x \text{ and } \nu(\alpha)\le A (\log_2{x})^{D} + \lambda \cdot B (\log_2{x})^{D-\frac12}\}}{\#\{(\alpha): 0 < |N(\alpha)|\le x\}} \to \int_{-\infty}^{\lambda} e^{-t^2/2}\, \mathrm{d}t,\]
as $x\to\infty$. Moreover, the constants $A$, $B$, and $D$ depend only on the isomorphism type of the class group of $K$.
\end{thm}
\noindent We can summarize Theorem \ref{thm:main} as asserting that $\nu(\alpha)$ has a normal distribution with mean $A(\log_2{|N(\alpha)|})^{D}$ and standard deviation $B(\log_2{|N(\alpha)|})^{D-\frac{1}{2}}$.

We say a little about the values of  $A$, $B$, and $D$. Of the three, $D$ is the simplest to describe: It is the smallest integer with the property that any sequence of $D$ elements of the class group $\Cl(\Z_K)$ contains a nonempty subsequence which multiplies to the identity. (If $G$ is any finite abelian group, the analogous quantity has become known as the \emph{Davenport constant} of $G$, and there is now a large literature on determining values of Davenport constants.) The appearance of $D$ in Theorem \ref{thm:main} is not so surprising. In fact, the constant $D$ is important to us for precisely the same reason it first caught the attention of Davenport: $D$ is the maximal number of prime ideals (counting multiplicity) that appear in the decomposition of an irreducible element of $\Z_K$.\footnote{According to Olson \cite{olson69}, Davenport reported this observation at the {M}idwestern conference on group theory
  and number theory, {O}hio {S}tate {U}niversity, April 1966.} The constants $A$ and $B$ are more complicated to define, but in the case when $\Cl(\Z_K)$ is cyclic of order $h$, we will show that $A= \phi(h) h^{-h} h!^{-1}$ and $B= h^{-h+3/2} h!^{-1} \phi(h)^{1/2}$.

A few words about strategy are in order. The function $\nu(\alpha)$ is not additive in any reasonable sense; even if $\alpha$ and $\beta$ generate comaximal ideals of $\Z_K$, we need not have $\nu(\alpha \beta) = \nu(\alpha)+\nu(\beta)$. For example, in $\Z[\sqrt{-5}]$, we have $\nu(2) = 1$ and $\nu(3)=1$, whereas $\nu(6)=4$. To work around this, we cook up  an additive function $f$ on $\Id(\Z_K)$ such that the behavior of $\nu$ is --- most of the time, and on the scale important for us --- determined by the distribution of $f$ restricted to $\Prin(\Z_K)$. We then study the distribution of $f|_{\Prin(\Z_K)}$ using the method of Granville--Soundararajan for proving Erd\H{o}s--Kac type theorems \cite{GS07}.

Of course, many other problems concerning $\nu$ could be investigated. We content ourselves with proving one additional result further reinforcing that ``irreducibles play a game of chance.''

\begin{thm}\label{thm:equidistribution} Fix $m \in \Z^{+}$. Then $\nu(\alpha)$ is equidistributed modulo $m$ as $\alpha$ ranges over $\Z_K$. More precisely, for each $a\in \Z$,
\[ \lim_{x\to\infty} \frac{\#\{(\alpha): 0 < |N(\alpha)| \le x,~\nu(\alpha)\equiv a\pmod{m}\}}{\#\{(\alpha): 0 < |N(\alpha)| \le x\}} = \frac{1}{m}. \]
\end{thm}
\noindent When $K=\Q$, this result is well-known (compare with \cite{selberg39}, \cite{pillai40}, \cite{addison57}). In fact, when $K=\Q$ and $m=2$, it goes back to von Mangoldt \cite{mangoldt97}; that case was later proved to be ``elementarily equivalent'' to the prime number theorem in work of Landau \cite{landau99,landau11}. (Actually, von Mangoldt and Landau deal with squarefree positive integers, but a convolution argument shows that the equidistribution assertion for squarefree integers is ``elementarily equivalent'' to the assertion for all positive integers.) Our proof of Theorem \ref{thm:equidistribution} is easily adapted to prove the  equidistribution mod $m$ of the count of prime ideal divisors of elements of $\Id(\Z_K)$ (this is again classical when $m=2$ \cite{landau03}), or of $\Prin(\Z_K)$; in fact, the arguments in these cases are much simpler.

Several further quantitative problems concerning factorizations in $\Prin(\Z_K)$ have been considered by Geroldinger, Halter-Koch, Kaczorowski, Narkiewicz, Odoni, R\'emond, \'Sliwa, and others. The interested reader is referred to the discussion in Chapter 9 of \cite{narkiewicz04} as well as the extensive end-of-chapter references there. See also \cite[Chapter 9]{GHK06}.

\section{Algebro-analytic input}
For the rest of this paper, $K$ is a degree $d$ number field admitting $r_1$ real embeddings and $r_2$ pairs of complex conjugate embeddings, so that $d=r_1+2r_2$. We let $h:=\#\Cl(\Z_K)$ denote the class number, $R$ the regulator, $\Delta$ the discriminant, and $w$ the number of roots of unity contained in $K$. We fix an ordering $\mathcal{C}_1, \dots, \mathcal{C}_h$ of the elements of $\Cl(\Z_K)$. Elements of $\Id(\Z_K)$ are generally indicated with Fraktur letters; $\mathfrak{p}$ and $\mathfrak{q}$ are reserved for nonzero prime ideals of $\Z_K$. Implied constants may always depend on $K$ without further mention.

The next two results are classical.
\begin{lem}\label{lem:weber} For each ideal class $\mathcal{C}$ of $\Z_K$, and all $x\ge 1$,
\[ \sum_{\substack{N\aa \le x \\ \aa \in \mathcal{C}}} 1 = \frac{ \Psi x }{h}+ O(x^{1-\frac1d}), \quad\text{where}\quad \Psi := \frac{2^{r_1+r_2} \pi^{s} R}{w \sqrt{|\Delta|}}. \]
\end{lem}
\begin{proof} This is due to Weber \cite{weber96}.
\end{proof}

\begin{lem}\label{lem:landau} For each ideal class $\mathcal{C}$ of $\Z_K$, and all $x\ge 3$,
\[ \sum_{\substack{N\pp \le x \\ \pp \in \mathcal{C}}} \frac{1}{|\pp|} = \frac{1}{h} \log_2{x} + O(1).\]
\end{lem}
\begin{proof} This follows from Landau's ideal class variant of the prime ideal theorem (with error term) \cite[Satz LXXXV]{landau18}, after partial summation.
\end{proof}

\section{Reduction to a standard Erd\H{o}s-Kac problem}
\subsection{Preliminary anatomical results} We begin by recording two easy consequences of the analytic lemmas recalled in the preceding section. For each $i=1,2,\dots, h$, let $\omega_i(\aa)$ denote the number of distinct prime ideal factors of $\aa$ from $\Cc_i$, and let $\Omega_i(\aa)$ denote the corresponding count with multiplicity. Then $\omega_i$ and $\Omega_i$ are additive functions on $\Id(\Z_K)$, in the sense that $\omega_i(\mathfrak{a}\mathfrak{b}) = \omega_i(\mathfrak{a}) + \omega_i(\mathfrak{b})$ for comaximal ideals $\mathfrak{a}$ and $\mathfrak{b}$, and similarly for $\Omega_i$.

\begin{prop}\label{prop:normal} For each $i=1,2,\dots, h$, and all $x\ge 3$, we have
\[ \sum_{N(\mathfrak{a}) \le x} \left(\omega_i(\mathfrak{a}) -\frac{1}{h}\log_2{x}\right)^2 = O(x\log_2{x}). \]
\end{prop}

\begin{prop}\label{prop:Omeganormal} For each $i=1,2,\dots, h$, and all $x\ge 3$, we have
\[ \sum_{N(\mathfrak{a}) \le x} \left(\Omega_i(\mathfrak{a})-\omega_i(\mathfrak{a})\right) = O(x). \]
\end{prop}

Propositions \ref{prop:normal} and \ref{prop:Omeganormal} follow by a straightforward imitation of the classical proofs for $K=\Q$ (when $h=1$), as found in Hardy and Wright \cite[see eqs. (22.10.1), (22.10.2), and (22.11.7)]{HW08}.
%The error estimates are slightly more difficult than in the traditional setting, but as the difficulties are precisely analogous to ones we will encounter later (in the proof of Theorem \ref{thm:EKstandard}), we omit the arguments.

\subsection{The \emph{type} of an irreducible and a decomposition of $\nu(\alpha)$} Let $\pi$ be an irreducible element of $\Z_K$. Suppose that the decomposition of $(\pi)$ into prime ideals takes the form
\begin{equation}\label{eq:pifact} (\pi) = \mathfrak{p}_1 \cdots \mathfrak{p}_g. \end{equation}
The irreducibility of $\pi$ guarantees that no nonempty, proper subsequence of $\mathfrak{p}_1, \dots, \mathfrak{p}_g$ multiplies to a principal ideal. Hence, $g\le D$, the Davenport constant of the class group $\Cl(\Z_K)$. On the other hand, if $\mathfrak{p}_1, \dots, \mathfrak{p}_g$ are prime ideals whose product is principal but no nonempty proper subproduct is principal, then $\mathfrak{p}_1\cdots\mathfrak{p}_g= (\pi)$ for an irreducible $\pi$. Since every ideal class contains prime ideals, one can construct irreducibles $\pi$ with $D$ prime ideal factors (counting multiplicity): choose $\pp_1,\dots,\pp_{D-1}$ having no nontrivial principal subproduct, and choose $\pp_{D}$ so that $\pp_{1}\cdots\pp_{D}$ is principal. Then each generator $\pi$ of $\pp_1\cdots\pp_{D}$ is irreducible. We thus recover Davenport's result that $D$ is the maximal number of prime ideals appearing in the decomposition of an irreducible element of $\Z_K$.

Define the \emph{type} $\tau$ of $\pi$ as the integer tuple $(t_1, \dots, t_h)$, where $t_i$ is the number of $\pp$ in \eqref{eq:pifact} belonging to $\Cc_i$, counted with multiplicity. Let $\mathcal{T}$ denote the set of types $\tau$ that correspond to some irreducible. For each $\tau=(t_1,\dots,t_h)\in \mathcal{T}$, we have $t_1 + \dots + t_h =g \le D$. When $t_1+\dots+t_h=D$, we say $\tau$ is of \emph{maximal length}.

For $\alpha \in \Z_K$ and $\tau \in \mathcal{T}$, we define $\nu_{\tau}(\alpha)$ as the number of distinct nonassociate irreducibles dividing $\alpha$ and having type $\tau$. Thus,
\[ \nu(\alpha) = \sum_{\tau \in \mathcal{T}} \nu_{\tau}(\alpha).\]
We now turn attention to the summands $\nu_{\tau}(\alpha)$.

Specifying a type-$\tau$ irreducible factor of $\alpha$ amounts to making $h$ choices: For each $1\le i \le h$, we must choose $t_i$ prime ideals (not necessarily distinct) from the multiset of prime ideals dividing $\alpha_i$  belonging to the class $\mathcal{C}_i$. Abusing notation somewhat and writing $\omega_i(\alpha)$ for $\omega_i((\alpha))$, and similarly for $\Omega_i$, the number of ways the $i$th choice can be made is bounded below by $\binom{\omega_i(\alpha)}{t_i}$ and bounded above by $\binom{\Omega_i(\alpha)}{t_i}$. Hence,
\begin{equation}\label{eq:nusqueeze} \prod_{i=1}^{h} \binom{\omega_i(\alpha)}{t_i} \le \nu_{\tau}(\alpha) \le \prod_{i=1}^{h}\binom{\Omega_i(\alpha)}{t_i}. \end{equation}

In order to obtain useful estimates from \eqref{eq:nusqueeze}, we will assume that $\alpha$ avoids a small exceptional set. Let $x\ge 3$, and let $\alpha$ be a nonzero element of $\Z_K$ with $|N(\alpha)|\le x$. It is convenient for what follows if $(\alpha)$ satisfies
\begin{enumerate}
\item $|\omega_i(\alpha)-\frac{1}{h}\log_2{x}| < (\log_2{x})^{2/3}$ for all $i=1,2,\dots,h$,
\item $|\Omega_i(\alpha)-\omega_i(\alpha)|<\log_3{x}$ for all $i=1,2,\dots,h$.
\end{enumerate}
Let $\mathcal{E}$ denote the set of principal ideals $(\alpha)$ of norm not exceeding $x$ for which one at least of (i) or (ii) fails. Propositions \ref{prop:normal} and \ref{prop:Omeganormal} imply that
\[ \#\mathcal{E} \ll x/\log_3{x}.\]
In particular, $\mathcal{E}$ makes up asymptotically 0\% of the the principal ideals of norm bounded by $x$, as $x\to\infty$.

Suppose that $(\alpha) \notin \mathcal{E}$. Using (i), we see that
\begin{align} \prod_{i=1}^{h} \binom{\omega_i(\alpha)}{t_i} &= \prod_{i=1}^{h} \left(\frac{\omega_i(\alpha)^{t_i}}{t_i!}\left(1 + O(1/\log_2{x})\right)\right) \notag\\
&= \left(\prod_{i=1}^{h} \frac{\omega_i(\alpha)^{t_i}}{t_i!}\right)\left(1 + O(1/\log_2{x})\right). \label{eq:wlower} \end{align}
On the other hand, (i) and (ii) together imply that $\Omega_i(\alpha)/\omega_i(\alpha)=1 + O(\log_3{x}/\log_2{x})$ for each $i=1,2,\dots, h$. Hence,
\begin{align}\notag \prod_{i=1}^{h} \binom{\Omega_i(\alpha)}{t_i} &= \prod_{i=1}^{h} \left(\frac{\Omega_i(\alpha)^{t_i}}{t_i!}\left(1 + O(1/\log_2{x})\right)\right) \\ &= \left(\prod_{i=1}^{h} \frac{\omega_i(\alpha)^{t_i}}{t_i!}\right)\left(1 + O(\log_3{x}/\log_2{x})\right).\label{eq:Wupper} \end{align}
If $\tau$ is not of maximal length, so that $t_1 + \dots + t_h \le D-1$, we deduce from the upper bound in \eqref{eq:nusqueeze} along with (i) and \eqref{eq:Wupper} that
\[ \nu_{\tau}(\alpha) = O((\log_2{x})^{D-1}). \]
Suppose now that $\tau$ is of maximal length. Then combining \eqref{eq:nusqueeze}, \eqref{eq:wlower}, and \eqref{eq:Wupper} with (i) reveals that
\begin{equation}\label{eq:firstnu} \nu_{\tau}(\alpha) = \prod_{i=1}^{h} \frac{\omega_i(\alpha)^{t_i}}{t_i!} + O((\log_2{x})^{D-1} \log_3{x}). \end{equation}
Write $\omega_i(\alpha) = \frac{1}{h}\log_2{x}\left(1 + \frac{\omega_i(\alpha)-\frac{1}{h}\log_2{x}}{\frac{1}{h}\log_2{x}}\right)$.
Then (keeping in mind (i))
\[ \frac{\omega_i(\alpha)^{t_i}}{t_i!} = \frac{1}{t_i !} \left(\frac{1}{h}\log_2{x}\right)^{t_i} \left(1+t_i \frac{\omega_i(\alpha)-\frac{1}{h}\log_2{x}}{\frac{1}{h}\log_2{x}} + O((\log_2{x})^{-2/3})\right).\]
Inserting this into \eqref{eq:firstnu},
\[ \nu_{\tau}(\alpha) = \left(\prod_{i=1}^{h}\frac{1}{t_i!}\right) \left(\frac{1}{h}\log_2{x}\right)^D\cdot \left(1+\frac{\sum_{j=1}^{h} t_j \omega_j(\alpha) -\frac{D}{h}\log_2{x}}{\frac{1}{h}\log_2{x}} + O((\log_2{x})^{-2/3})\right).\]
Now sum on $\tau\in\mathcal{T}$. To keep track of the components of the various $\tau$, instead of $t_1,\dots,t_h$, we switch notation to $t_{1}(\tau), \dots, t_{h}(\tau)$. Then
\begin{multline} \nu(\alpha) = \Bigg(\sum_{\substack{\tau \in \mathcal{T} \\ \tau\text{ maximal}}} \prod_{i=1}^{h}\frac{1}{t_{i}(\tau)!}\Bigg) \left(\frac{1}{h}\log_2{x}\right)^D \\ + \left(\frac{1}{h}\log_2 x\right)^{D-1} \left(\sum_{j=1}^{h} \Bigg(\sum_{\substack{\tau\in \mathcal{T}\\ \tau \text{ maximal}}} t_j(\tau)\prod_{i=1}^{h}\frac{1}{t_{i}(\tau)!}\Bigg) \omega_j(\alpha) - \Bigg(\frac{D}{h}\sum_{\substack{\tau \in \mathcal{T} \\ \tau\text{ maximal}}}\prod_{i=1}^{h}\frac{1}{t_{i}(\tau)!}\Bigg) \log_2{x}\right)\\
+ O((\log_2{x})^{D-2/3}).\label{eq:longnu} \end{multline}
To continue, for $1\le j\le h$, set
\begin{equation}\label{eq:kappadef} \kappa_j = \sum_{\substack{\tau\in \mathcal{T}\\ \tau \text{ maximal}}} t_j(\tau)\prod_{i=1}^{h}\frac{1}{t_{i}(\tau)!}. \end{equation}
Then
\[ \sum_{j=1}^{h} \kappa_j = \sum_{\substack{\tau\in \mathcal{T}\\ \tau \text{ maximal}}} \prod_{i=1}^{h}\frac{1}{t_{i}(\tau)!} \sum_{j=1}^{h} t_j(\tau) = D \sum_{\substack{\tau \in \mathcal{T} \\ \tau\text{ maximal}}} \prod_{i=1}^{h}\frac{1}{t_{i}(\tau)!}. \]

In the next section, we will prove the following Erd\H{o}s--Kac type result for certain additive functions $f$ on $\Id(\Z_K)$, restricted to $\Prin(\Z_K)$.

\begin{thm}\label{thm:EKstandard} Let $\kappa_1,\dots, \kappa_h$ be nonnegative constants, not all of which vanish. For nonzero $\alpha \in \Z_K$, let
\[ f(\alpha) = \sum_{j=1}^{h} \kappa_j \omega_j(\alpha). \]
As $x\to\infty$, the quantity
\[ \frac{f(\alpha)-\left(\frac{1}{h}\sum_{j=1}^{h} \kappa_j\right)\log_2{x}}{\sqrt{(\frac{1}{h}\sum_{j=1}^{h}\kappa_j^2)\log_2{x}}}, \] considered as a random variable on the space of nonzero principal ideals $(\alpha)$ of norm $\le x$ (with the uniform measure),  converges in law to a standard normal  distribution.
\end{thm}

Theorem \ref{thm:main} follows easily from Theorem \ref{thm:EKstandard}. Indeed, from \eqref{eq:longnu}, we have that when $(\alpha)\notin \mathcal{E}$,
\begin{multline*} \frac{\nu(\alpha) - \left(\sum_{\substack{\tau \in \mathcal{T} \\ \tau\text{ maximal}}} \prod_{i=1}^{h}\frac{1}{t_{i}(\tau)!}\right) \left(\frac{1}{h}\log_2{x}\right)^D}{(\frac{1}{h}\log_2{x})^{D-1} \sqrt{(\frac{1}{h}\sum_{j=1}^{h}\kappa_j^2)\log_2{x}}} \\ = \frac{\sum_{j=1}^{h}\kappa_j \omega_j(\alpha) - \left(\frac{1}{h}\sum_{j=1}^{h}\kappa_j\right)\log_2{x}}{\sqrt{(\frac{1}{h}\sum_{j=1}^{h}\kappa_j^2)\log_2{x}}}+O((\log_2{x})^{-1/6}).\end{multline*}
Since only $o(x)$ ideals $(\alpha)$ land in $\mathcal{E}$, and $(\log_2{x})^{-1/6} = o(1)$, Theorem \ref{thm:EKstandard} implies Theorem \ref{thm:main} with
\[ A = \frac{1}{h^D} \sum_{\substack{\tau \in \mathcal{T} \\ \tau\text{ maximal}}} \prod_{i=1}^{h}\frac{1}{t_{i}(\tau)!} \qquad \text{i.e.},\qquad A = \frac{1}{D h^D} \sum_{j=1}^{h} \kappa_j,\]
and
\[ B = \frac{1}{h^{D-1/2}} \sqrt{\sum_{j=1}^{h} \kappa_j^2}.\]

\begin{example}[Calculation of $A$ and $B$ when the class group is cyclic; cf. {\cite[\S4]{BORS05}}] Suppose that $\Cl(\Z_K)$ is a cyclic group of order $h$. Then $D=h$. (See \cite[\S9.1]{narkiewicz04} for basic facts about Davenport constants.) We suppose the ideal classes are numbered so that, under a fixed isomorphism of $\Cl(\Z_K)$ with $\Z/h\Z$, the class $\mathcal{C}_i$ corresponds to $i\bmod{h}$. Then there are $\phi(h)$ types $\tau$ of maximum length, namely $(0,,\dots,0,h,0,\dots,0)$, where the allowed positions for `$h$' are precisely the units mod $h$ (compare with \cite[Corollary 2.1.4, p. 24]{GR09}). From \eqref{eq:kappadef}, we see that $\kappa_j = \frac{1}{(h-1)!}$ when $\gcd(j,h)=1$ and $\kappa_j=0$ otherwise. After a bit of algebra, we arrive at
\[ A = \frac{1}{h^{h} h!} \phi(h) \qquad\text{and} \qquad B= \frac{1}{h^{h} h!} (h^3 \phi(h))^{1/2}, \]
as claimed in the introduction.
\end{example}

\section{Proof of Theorem \ref{thm:EKstandard}}
To prove Theorem \ref{thm:EKstandard}, we follow very closely the approach to the Erd\H{o}s--Kac theorem detailed by Granville and Soundararajan \cite{GS07}. By the method of moments (the Fr\'echet-Shohat theorem), to prove Theorem \ref{thm:EKstandard} it is sufficient to establish the following estimates.

\begin{prop}\label{prop:moments} Let $k$ be a fixed positive integer. Suppose that $x$ is sufficiently large. If $k$ is even, then
\begin{multline*} \sum_{\substack{(\alpha) \\ 0 <|N(\alpha)| \le x}} \left(f(\alpha)-\bigg(\frac{1}{h}\sum_{j=1}^{k}\kappa_j\bigg)\log_2{x}\right)^k \\ = \frac{\Psi x}{h} \cdot 	\frac{k!}{2^{k/2} \cdot \frac{k}{2}!}  \left(\frac{1}{h}\sum_{j=1}^{h}\kappa_j^2\right)^{k/2}(\log_2{x})^{k/2} + O(x (\log_2{x})^{\frac{k-1}{2}}). \end{multline*}
If $k$ is odd, then
\[ \sum_{\substack{(\alpha)\\ 0 <|N(\alpha)| \le x}} \left(f(\alpha)-\bigg(\frac{1}{h}\sum_{j=1}^{k}\kappa_j\bigg)\log_2{x}\right)^k = O((\log_{2}x)^{\frac{k-1}{2}}). \]
Here the implied constants may depend not only on $K$ but also on $k$ and the $\kappa_j$.
\end{prop}

Proposition \ref{prop:moments} will be deduced from the following technical lemma. For a nonzero prime ideal $\pp$ of $\Z_K$, we set $\kappa_{\pp} = \kappa_j$, where $j$ is that index for which $\pp \in \Cc_j$.

\begin{lem}\label{lem:technical} For each nonzero prime ideal $\mathfrak{p}$ of $\Z_K$, and each nonzero ideal $\mathfrak{a}$ of $\Z_K$, set
\[ g_\pp(\mathfrak{a}) = \begin{cases} 1-\frac{1}{N\mathfrak{p}}&\text{if $\mathfrak{p} \mid \mathfrak{a}$},\\
-\frac{1}{N\pp} &\text{if $\mathfrak{p}\nmid\mathfrak{a}$}.\end{cases} \]
Fix a positive integer $k$. Suppose $x$ is sufficiently large, and let $z= x^{\frac{1}{2dk}}$. If $k$ is even, then
\[ \sum_{\substack{(\alpha)\\ 0<|N(\alpha)| \le x}} \bigg(\sum_{N\pp \le z} \kappa_{\pp} g_{\pp}(\alpha)\bigg)^k = \frac{\Psi x}{h} \cdot 	\frac{k!}{2^{k/2} \cdot \frac{k}{2}!}  \left(\frac{1}{h}\sum_{j=1}^{h}\kappa_j^2\right)^{k/2}(\log_2{x})^{k/2} + O(x (\log_2{x})^{\frac{k}{2}-1}). \]
If $k$ is odd, then
\[ \sum_{\substack{(\alpha)\\ 0 < |N(\alpha)| \le x}} \bigg(\sum_{N\pp \le z} \kappa_{\pp} g_{\pp}(\alpha)\bigg)^k = O((\log_{2}x)^{\frac{k-1}{2}}).\]
Here the implied constants may depend on $K$, $k$, and the $\kappa_j$.
\end{lem}

\begin{proof}[Deduction of Proposition \ref{prop:moments} from Lemma \ref{lem:technical}] For each $j$, let $\omega_j(\alpha;z)$ denote the number of prime ideal factors $\pp$ of $(\alpha)$ with $\pp \in \Cc_j$ and $N\pp \le z$. Using Lemma \ref{lem:landau}, we see that
\begin{align*} \sum_{N\pp \le z} \kappa_\pp g_\pp(\alpha) = \sum_{\substack{\pp \mid \alpha \\ N\pp \le z}} \kappa_\pp - \sum_{N\pp \le z}\frac{\kappa_\pp}{N\pp} &= \sum_{j=1}^{h} \kappa_j \omega_j(\alpha;z) - \sum_{j=1}^{h} \kappa_j\left(\frac{1}{h}\log_{2}z + O(1)\right) \\
&= f(\alpha) - \bigg(\frac{1}{h}\sum_{j=1}^{h}\kappa_j\bigg)\log_2{x} + O(1).\end{align*}
To go from the first line to the second, we used that $\log_2{z} = \log_2{x}+O(1)$ and that $(\alpha)$ can have only $O(1)$ prime ideal factors of norm exceeding $z$.
We deduce that
\[ \left(f(\alpha) - \bigg(\frac{1}{h}\sum_{j=1}^{h}\kappa_j\bigg)\log_2{x}\right)^k = \left(\sum_{N\pp \le z} \kappa_\pp g_\pp(\alpha)\right)^k + O\left(\sum_{\ell=0}^{k-1}\bigg|\sum_{N\pp\le z}\kappa_p g_\pp(\alpha)\bigg|^{\ell}\right).\]
Sum both sides over principal ideals $(\alpha)$ with $0 < |N(\alpha)| \le x$. To estimate the main term on the right-hand side, we may appeal to Lemma \ref{lem:technical}. We can also use Lemma \ref{lem:technical} to see that the even values of $\ell$ make an acceptable contribution to the error term. If $\ell$ is odd, we use Cauchy--Schwarz to deduce that
\begin{multline*} \sum_{(\alpha):~0 < |N(\alpha)|\le x} \bigg|\sum_{N\pp\le z}\kappa_p g_\pp(\alpha)\bigg|^{\ell} \\ \le \left(\sum_{(\alpha):~0 < |N(\alpha)|\le x} \bigg|\sum_{N\pp\le z}\kappa_p g_\pp(\alpha)\bigg|^{\ell-1}\right)^{1/2} \left(\sum_{(\alpha):~0 < |N(\alpha)|\le x} \bigg|\sum_{N\pp\le z}\kappa_p g_\pp(\alpha)\bigg|^{\ell+1}\right)^{1/2}.\end{multline*}
Appealing once more to Lemma \ref{lem:technical}, we find that the contribution of the odd $\ell$ also fits within the $O$-term claimed in Proposition \ref{prop:moments}.
\end{proof}

\begin{proof}[Proof of Lemma \ref{lem:technical}] If $\mathfrak{r} = \prod_{i} \pp_i^{e_i}$, put $g_{\mathfrak{r}}(\mathfrak{a}) = \prod_{i} g_{\pp}(\mathfrak{a})^{e_i}$. Then
\begin{equation}\label{eq:initial} \sum_{\substack{(\alpha)\\ 0 < |N(\alpha)|\le x}} \bigg(\sum_{N\pp\le z}\kappa_p g_\pp(\alpha)\bigg)^{k} = \sum_{\substack{\pp_1, \dots, \pp_k\\\text{each $N\pp_i \le z$}}} \kappa_{\pp_1} \cdots \kappa_{\pp_k} \sum_{\substack{(\alpha)\\ 0 < |N(\alpha)|\le x}} g_{\pp_1 \cdots \pp_k}(\alpha). \end{equation}
To proceed, we consider more generally sums of the form \[ \sum_{\substack{(\alpha)\\ 0 < |N(\alpha)|\le x}} g_{\mathfrak{r}}(\alpha),\]
for any $\mathfrak{r}$ with $N\mathfrak{r} \le z^k$. Write $\mathfrak{r}=\mathfrak{q}_1^{e_1} \cdots \mathfrak{q}_s^{e_s}$, where the $\mathfrak{q}_i$ are distinct prime ideals. Put $\mathfrak{R} = \mathfrak{q}_1 \cdots \mathfrak{q}_s$. If $\mathfrak{d} = \gcd((\alpha), \mathfrak{R})$, then $g_{\mathfrak{r}}(\alpha) = g_{\mathfrak{r}}(\mathfrak{d})$. Hence,
\begin{equation}\label{eq:complicatedinnersum} \sum_{\substack{(\alpha)\\ 0 < |N(\alpha)|\le x}} g_{\mathfrak{r}}(\alpha) = \sum_{\mathfrak{d} \mid \mathfrak{R}} g_{\mathfrak{r}}(\mathfrak{d})\sum_{\substack{(\alpha)\\ 0 < |N(\alpha)|\le x \\ \gcd((\alpha),\mathfrak{R}) = \mathfrak{d}}} 1.\end{equation}
We turn attention to the right-hand inner sum. Observe that $\gcd((\alpha),\mathfrak{R})=\mathfrak{d}$ precisely when $\mathfrak{d} \mid \alpha$ and $\alpha \mathfrak{d}^{-1}$ and $\mathfrak{R} \mathfrak{d}^{-1}$ are coprime. Thus, thinking of $\bb = \alpha \mathfrak{d}^{-1}$, the inner sum equals
\[ \sum_{\substack{\bb:~N\bb \le x/N\mathfrak{d}\\ [\mathfrak{b}] = [\mathfrak{d}]^{-1} \\ \gcd(\mathfrak{b}, \mathfrak{R}\mathfrak{d}^{-1}) = 1}} 1. \]
(Here and below, $[\cdot]$ denotes the image of an ideal in the class group $\Cl(\Z_K)$.)
Letting $\mu$ denote the M\"{o}bius function on $\Id(\Z_K)$,
\[ \sum_{\substack{\bb:~N\bb \le x/N\mathfrak{d}\\ [\mathfrak{b}] = [\mathfrak{d}]^{-1} \\ \gcd(\mathfrak{b}, \mathfrak{R}\mathfrak{d}^{-1}) = 1}} 1 =
\sum_{\substack{\bb:~N\bb \le x/N\mathfrak{d}\\ [\mathfrak{b}] = [\mathfrak{d}]^{-1}}} \sum_{\substack{\mathfrak{e} \mid \mathfrak{R} \mathfrak{d}^{-1} \\ \mathfrak{e} \mid \mathfrak{b}}} \mu(\mathfrak{e}) = \sum_{\mathfrak{e} \mid \mathfrak{R} \mathfrak{d}^{-1}} \mu(\mathfrak{e}) \sum_{\substack{\bb:~N\bb \le x/N\mathfrak{d}\\ [\mathfrak{b}] = [\mathfrak{d}]^{-1} \\ \mathfrak{e} \mid \mathfrak{b}}} 1.
\]
Writing $\mathfrak{b}=\mathfrak{e}\mathfrak{f}$, we see that
\begin{align*} \sum_{\mathfrak{e} \mid \mathfrak{R} \mathfrak{d}^{-1}} \mu(\mathfrak{e}) \sum_{\substack{\bb:~N\bb \le x/N\mathfrak{d}\\ [\mathfrak{b}] = [\mathfrak{d}]^{-1} \\ \mathfrak{e} \mid \mathfrak{b}}} 1 &= \sum_{\mathfrak{e} \mid \mathfrak{R} \mathfrak{d}^{-1}} \mu(\mathfrak{e}) \sum_{\substack{\mathfrak{f}:~N\mathfrak{f}\le x/N(\mathfrak{d}\mathfrak{e}) \\ [\mathfrak{f}] = [\mathfrak{d}\mathfrak{e}]^{-1}}} 1 \\
&= \sum_{\mathfrak{e} \mid \mathfrak{R} \mathfrak{d}^{-1}} \left(\frac{\Psi x}{h} \frac{\mu(\mathfrak{e})}{N\mathfrak{d} N\mathfrak{e}} + O\left(\left(\frac{x}{N(\mathfrak{d}\mathfrak{e})}\right)^{1-1/d}\right)\right)\\
&= \frac{\Psi x}{h \cdot N\mathfrak{d}} \cdot \frac{\phi(\mathfrak{\mathfrak{R} \mathfrak{d}^{-1}})}{N(\mathfrak{R} \mathfrak{d}^{-1})} + O\left(x^{1-1/d} \sum_{\mathfrak{e} \mid \mathfrak{R} \mathfrak{d}^{-1}} \frac{1}{(N(\mathfrak{d}\mathfrak{e}))^{1-1/d}}\right). \end{align*}
(Here the ideal-theoretic $\phi$-function is defined by $\phi(\mathfrak{u}) =\#(\Oo/\mathfrak{u})^{\times}$.)
Plugging this back into \eqref{eq:complicatedinnersum}, and using that $|g_{\mathfrak{r}}(\mathfrak{d})| \le 1$,
\[ \sum_{\substack{(\alpha)\\ 0 < |N(\alpha)|\le x}} g_{\mathfrak{r}}(\alpha) = \frac{\Psi x}{h} \cdot \frac{1}{N\mathfrak{R}} \sum_{\mathfrak{d} \mid \mathfrak{R}} g_{\mathfrak{r}}(\mathfrak{d}) \phi(\mathfrak{R} \mathfrak{d}^{-1}) + O\Bigg(x^{1-1/d}\sum_{\substack{\mathfrak{d}, \mathfrak{e} \\ \mathfrak{d}\mathfrak{e} \mid \mathfrak{R}}} \frac{1}{N(\mathfrak{de})^{1-1/d}}\Bigg).\]
The error term here is harmless: For any $\epsilon >0$, there are $\ll_{\epsilon} N(\mathfrak{R})^{\epsilon} \le z^{k\epsilon} = x^{\epsilon/2d}$ ideal divisors of $\mathfrak{R}$. Hence, the number of pairs $\mathfrak{d}, \mathfrak{e}$ with $\mathfrak{d}\mathfrak{e}$ dividing $\mathfrak{R}$ is crudely $\ll_{\epsilon} x^{\epsilon/d}$. Trivially, $1/N(\mathfrak{d}\mathfrak{e})^{1-1/d} \le 1$, and so we see (taking $\epsilon =1/4$) that the $O$-term above is $O(x^{1-\frac{3}{4d}})$. The sum on $\mathfrak{d}$ dividing $\mathfrak{R}$ appearing in the main term can be explicitly evaluated; we find that
\begin{equation}\label{eq:explicitevaluation} \sum_{\substack{(\alpha)\\ 0 < |N(\alpha)|\le x}} g_{\mathfrak{r}}(\alpha) = \frac{\Psi x}{h} \cdot G(\mathfrak{r}) + O(x^{1-\frac{3}{4d}}),  \end{equation}
where
\[ G(\mathfrak{r}) := \prod_{\mathfrak{q}^e\parallel \mathfrak{r}} \left(\frac{1}{N\mathfrak{q}}\left(1-\frac{1}{N\mathfrak{q}}\right)^{e} + \left(\frac{-1}{N\mathfrak{q}}\right)^{e}\left(1-\frac{1}{N\mathfrak{q}}\right)\right). \]
We see from this formula that $G(\mathfrak{r})$ vanishes unless $\mathfrak{r}$ is squarefull, by which we mean that  each prime ideal divisor of $\mathfrak{r}$ is repeated.

Substituting \eqref{eq:explicitevaluation} back into \eqref{eq:initial},
\[ \sum_{\substack{(\alpha)\\ 0 < |N(\alpha)|\le x}} \bigg(\sum_{N\pp\le z}\kappa_p g_\pp(\alpha)\bigg)^{k} = \frac{\Psi x}{h} \sum_{\substack{\pp_1, \dots, \pp_k \\ \text{each }N\pp_i \le z \\ \pp_1 \cdots \pp_k\text{ squarefull}}} \kappa_{\pp_1} \cdots \kappa_{\pp_k} G(\pp_1 \cdots \pp_k) + O(x^{1-\frac{3}{4d}} z^{k}).\]
The $O$-term is $\ll x^{1-\frac{1}{4d}}$, which will be negligible for us. The main term can be rewritten as
\[ \frac{\Psi x}{h} \sum_{s \le k/2} \frac{1}{s!}\sum_{\substack{\mathfrak{q}_1, \dots, \mathfrak{q}_s \\ \mathfrak{q}_i\text{ distinct}\\\text{each }N\mathfrak{q}_i\le z}} \sum_{\substack{e_1,\dots, e_s \ge 2 \\ \sum e_i = k}} \frac{k!}{e_1! \cdots e_s!} \cdot \kappa_{\mathfrak{q}_1}^{e_1} \cdots \kappa_{\mathfrak{q}_s}^{e_s} \cdot G(\mathfrak{q}_1^{e_1} \cdots \mathfrak{q}_s^{e_s}).\]

Let us estimate the contribution from the values of $s < k/2$. We use the easily seen inequality $|G(\mathfrak{q}_1^{e_1} \cdots \mathfrak{q}_s^{e_s})| \le \frac{1}{N(\mathfrak{q}_1 \cdots \mathfrak{q}_s)}$ to deduce that
\begin{multline*} \frac{\Psi x}{h} \sum_{s < k/2} \frac{1}{s!}\sum_{\substack{\mathfrak{q}_1, \dots, \mathfrak{q}_s \\ \mathfrak{q}_i\text{ distinct}\\ \text{each }N\mathfrak{q}_i\le z}} \sum_{\substack{e_1,\dots, e_s \ge 2 \\ \sum e_i = k}} \frac{k!}{e_1! \cdots e_s!} \cdot \kappa_{\mathfrak{q}_1}^{e_1} \cdots \kappa_{\mathfrak{q}_s}^{e_{s}} \cdot G(\mathfrak{q}_1^{e_1} \cdots \mathfrak{q}_s^{e_s}) \\ \ll x \sum_{s < k/2} \sum_{\substack{\mathfrak{q}_1,\dots, \mathfrak{q}_s \\\text{each $N\mathfrak{q}_i \le z$}}} \frac{1}{N(\mathfrak{q}_1\cdots \mathfrak{q}_s)} = x \sum_{s < k/2} \left(\sum_{N\mathfrak{q}\le z}\frac{1}{N\mathfrak{q}}\right)^s \ll x(\log_2{x})^{\lfloor (k-1)/2\rfloor}, \end{multline*}
since each $s$ in the sum is at most $\frac{k-1}{2}$ and $\sum_{N\mathfrak{q}\le z}\frac{1}{N\mathfrak{q}} \le \log_2{x} + O(1)$. Thus, the values $s < k/2$ contribute $\ll x(\log_2{x})^{\frac{k-1}{2}}$ when $k$ is odd and $\ll x(\log_2{x})^{\frac{k}{2}-1}$ when $k$ is even. This completes the proof of Lemma \ref{lem:technical} in the odd $k$ case.

When $k$ is even, there is an additional contribution corresponding to $s=k/2$ and $e_1=e_2=\dots=e_{k/2} = 2$, of size
\[ \frac{\Psi x}{h} \cdot \frac{k!}{2^{k/2} \cdot \frac{k}{2}!} \sum_{\substack{\mathfrak{q}_1, \dots, \mathfrak{q}_{k/2} \\ \mathfrak{q}_i\text{ distinct}\\\text{each }N\mathfrak{q}_i\le z}} \prod_{i=1}^{k/2} \frac{\kappa_{\mathfrak{q}_i}^2}{N\mathfrak{q}_i}\left(1-\frac{1}{N\mathfrak{q}_i}\right).\]
Forgetting the distinctness restriction, we obtain an upper bound on this last sum of
\begin{align*} \Bigg(\sum_{N\mathfrak{q}\le z} \frac{\kappa_{\mathfrak{q}}^2}{N\mathfrak{q}}\Bigg)^{k/2} = \Bigg(\sum_{j=1}^{h} \kappa_j^2 \sum_{\substack{N\mathfrak{q}\le z \\ \mathfrak{q} \in \Cc_j}} \frac{1}{{N\mathfrak{q}}}\Bigg)^{k/2} &\le \Bigg(\bigg(\frac{1}{h}\sum_{j=1}^{h}\kappa_j^2\bigg) \log_2{x} + O(1)\Bigg)^{k/2} \\
&= \Bigg(\frac{1}{h}\sum_{j=1}^{h}\kappa_j^2\Bigg)^{k/2} (\log_2{x})^{k/2} + O((\log_2{x})^{\frac{k}{2}-1}). \end{align*} It is easy to obtain a lower bound of the same form. Indeed, for any given choices of $\mathfrak{q}_1, \dots, \mathfrak{q}_{\frac{k}{2}-1}$,
\begin{align*} \sum_{\substack{N\mathfrak{q}\le z \\\mathfrak{q}\ne \mathfrak{q}_1,\dots,\mathfrak{q}_{\frac{k}{2}-1} }} \frac{\kappa_{\mathfrak{q}}^2}{N\mathfrak{q}}\left(1-\frac{1}{N\mathfrak{q}}\right) &\ge \sum_{N\mathfrak{q}\le z} \frac{\kappa_{\mathfrak{q}}^2}{N\mathfrak{q}} + O(1) = \sum_{j=1}^{h} \kappa^2_j \sum_{\substack{N\mathfrak{q} \le z \\ \mathfrak{q} \in \Cc_j}} \frac{1}{N\mathfrak{q}} + O(1) \\
&= \left(\frac{1}{h}\sum_{j=1}^{h}\kappa_j^2\right)\log_2{z} + O(1) = \left(\frac{1}{h}\sum_{j=1}^{h}\kappa_j^2\right)\log_2{x} + O(1).\end{align*}
Repeating this procedure, we eventually find that
\begin{align*} \sum_{\substack{\mathfrak{q}_1, \dots, \mathfrak{q}_{k/2} \\ \mathfrak{q}_i\text{ distinct}\\\text{each }N\mathfrak{q}_i\le z}} \prod_{i=1}^{h} \frac{\kappa_{\mathfrak{q}_i}^2}{N\mathfrak{q}_i}\left(1-\frac{1}{N\mathfrak{q}_i}\right) &\ge \left(\left(\frac{1}{h}\sum_{j=1}^{h}\kappa_j^2\right)\log_2{x} + O(1)\right)^{k/2} \\
&=\left(\frac{1}{h}\sum_{j=1}^{h}\kappa_j^2\right)^{k/2} (\log_2{x})^{k/2} + O((\log_2{x})^{\frac{k}{2}-1}).
\end{align*}
Combining these estimates with the results of the preceding paragraph completes the proof in the even $k$ case.
\end{proof}

\begin{remark} Theorem \ref{thm:EKstandard} could also be proved by applying the results of \cite[Chapter 2]{dekroon66}, with  Lemma \ref{lem:landau} used as the analytic input. De Kroon proves an Erd\H{o}s--Kac theorem for $f|_{\Prin(\Z_K)}$, for additive functions $f$ on $\Id(\Z_K)$ satisfying conditions analogous to those in the main theorem of \cite{EK40}. De Kroon's approach follows \cite{EK40}, in that the central limit theorem is used in combination with Brun's sieve. We have preferred to give a more self-contained proof illustrating the flexibility of the method of \cite{GS07}.
\end{remark}

\section{Remarks on Theorem \ref{thm:main}}
\subsection{A $\Prin(\Z_K)$-analogue of an observation of Kac}
Let $d(n)$ denote the classical divisor function. In 1941, Kac \cite{kac41} showed that $\frac{\log d(n)}{\log 2}$ is normally distributed with mean and variance $\log_2{n}$. This can be proved by the following simple argument: According to \cite{EK40}, both $\omega(n)$ and $\Omega(n)$ are normally distributed with mean and variance $\log_2{n}$; now observe that $2^{\omega(n)} \le d(n) \le 2^{\Omega(n)}$ for all $n$.

Substituting the results of \cite{liu04} for those of \cite{EK40}, an identical argument shows that the divisor function on $\Id(\Z_K)$ --- which by abuse of notation we will also denote $d$ --- is normally distributed with mean and variance $\log_2 N(\aa)$. Seeking a $\Prin(\Z_K)$ analogue, define $\delta(\alpha)$ for nonzero $\alpha \in \Z_K$ as the number of nonassociate divisors of $\alpha$. It turns out that $\delta$ is distributed in the same way as $d$, by which we mean that $\frac{\log \delta(\alpha)}{\log 2}$ is normal with mean and variance $\log_2 |N(\alpha)|$.

Let us sketch the proof. First, one shows that $\omega(\alpha):=\sum_{i=1}^{h}\omega_i(\alpha)$ and $\Omega(\alpha):=\sum_{i=1}^{h}\Omega_i(\alpha)$ are both normally distributed with mean and variance $\log_2|N(\alpha)|$. For $\omega$, this follows from Theorem \ref{thm:EKstandard} with $\kappa_1 = \dots = \kappa_h= 1$. The $\Omega$ assertion then follows from Proposition \ref{prop:Omeganormal}. (For these claims, one could also appeal to \cite{dekroon66}.) Next, one observes that  $\delta(\alpha) \le d(\alpha) \le 2^{\Omega(\alpha)}$. On the other hand, one can construct many nonassociate divisors of $\alpha$ by the following recipe: For each $i=1,2,\dots,h$, list the distinct prime ideal divisors of $\alpha$ belonging to $\Cc_i$, choose a subset of these whose cardinality is a multiple of $h$, and then multiply the prime ideals from each of these subsets. This construction yields
\[ \delta(\alpha) \ge \prod_{i=1}^{h} \Bigg(\sum_{\substack{0 \le j \le \omega_i(\alpha) \\ h \mid j}}\binom{\omega_i(\alpha)}{j} \Bigg). \]
For each $i$, the sum on $j$ is $\gg 2^{\omega_i(\alpha)}$. (One can see this by rewriting the sum as $\frac{1}{h}\sum_{\zeta}(1+\zeta)^{\omega_i(\alpha)}$, where $\zeta$ runs over the $h$th roots of unity, and noting that $\zeta=1$ dominates.) It follows that
\[ \delta(\alpha) \gg 2^{\omega_1(\alpha) + \dots + \omega_h(\alpha)} = 2^{\omega(\alpha)}. \]
Thus,
\[ 2^{\omega(\alpha)+O(1)} \le \delta(\alpha)\le 2^{\Omega(\alpha)}, \] and now we can obtain the normal distribution result for $\delta$ in the same way as for $d$.

\subsection{Average order results}
One can show that as $x\to\infty$, the average of $\nu(\alpha)$ on principal ideals $(\alpha)$ of norm $\le x$ is $\sim A(\log\log{x})^D$ (for the same constant $A$ from Theorem \ref{thm:main}), while the corresponding average of $\delta(\alpha)$ is $\sim \frac{\Psi}{h}\log{x}$. These estimates are considerably simpler to prove than the normal order results discussed above. Indeed, they follow more or less immediately from
\[ \sum_{\substack{(\pi):~|N(\pi)| \le x \\ \pi\text{ irreducible}}} \frac{1}{|N(\pi)|} \sim A(\log\log{x})^{D} \quad\text{and}\quad   \sum_{(\alpha):~0<|N(\alpha)| \le x} \frac{1}{|N(\alpha)|} \sim \frac{\Psi}{h}\log{x}. \] The second of these may be derived quickly by partial summation from Lemma \ref{lem:weber}. The first can be proved by methods discussed earlier in this article, or by partial summation in conjunction with R\'emond's asymptotic formula for the count of nonassociate irreducibles of bounded norm \cite[Th\'eor\`eme II, pp. 391--392, and Corollaire, pp. 409--410]{remond66}.

\section{Equidistribution of $\nu(\alpha)$ in residue classes: \\Proof of Theorem \ref{thm:equidistribution}}

Below, we say a multiplicative function $f$ on $\Id(\Z_K)$ is of \emph{finite order} if for some positive integer $r$, all of the nonzero values of $f$ are $r$th roots of unity.

\begin{lem}\label{lem:hall} Let $f$ be a multiplicative function on $\Id(\Z_K)$ of finite order. Suppose that
\begin{equation}\label{eq:wcond}\sum_{f(\pp)\ne 1} \frac{1}{N\pp} \quad\text{diverges}. \end{equation}
Then $f$ has mean value $0$ on $\Id(\Z_K)$, in the sense that
\[ \sum_{N\aa \le x} f(\aa) = o(x), \]
as $x\to\infty$.
\end{lem}

\begin{proof} This appears to be well-known when $K=\Q$, and the argument in the general case is the same; for lack of a suitable reference we sketch the proof. We use a theorem of Hal\'asz \cite{halasz68}, as generalized to ``arithmetic semigroups'' (a setting which includes $\Id(\Z_K)$) by Lucht and Reifenrath \cite[Theorem 6.1]{LR01}. It suffices to show that for each real number $t$, the series
\begin{equation}\label{eq:todiverge} \sum_{\pp} \frac{1-\Re(f(\pp) \cdot N(\pp)^{-it})}{N\pp} \end{equation}
diverges. The divergence when $t=0$ follows quickly from \eqref{eq:wcond}. Now suppose that $t\ne 0$. Fix a positive integer $r$ such that $f(\Id(\Z_K))$ is contained in $\{0\}\cup\{\zeta: \zeta^r=1\}$. The prime ideal theorem implies that for $\gg_{r} x/\log{x}$ prime ideals of norm not exceeding $x$, the quantity $N(\pp)^{it}$ lies at a distance $\gg_{r} 1$ from each $r$th root of unity; the divergence of \eqref{eq:todiverge} is an easy consequence.\end{proof}

\begin{lem}\label{lem:idealclass} Let $f$ be a multiplicative function on $\Id(\Z_K)$ of finite order. Suppose that
\[ \sum_{\substack{\pp \text{ principal}\\ f(\pp) \ne 1}} \frac{1}{N\pp} \quad\text{diverges.} \]
Then $f$ has mean value 0 along each ideal class, in the sense that for each $i=1,2,\dots, h$,
\[ \sum_{\substack{N\aa \le x \\ \aa \in \Cc_i}} f(\aa) = o(x),\]
as $x\to\infty$.
\end{lem}
\begin{proof} By the orthogonality relations for group characters, it suffices to show that for each character $\chi$ of the class group $\Cl(\Z_K)$, the function $\chi f$ has mean value $0$ on $\Id(\Z_K)$. 	This follows immediately from Lemma \ref{lem:hall}.
\end{proof}

\begin{proof}[Proof of Theorem \ref{thm:equidistribution}] Throughout this proof, we assume that the ideal classes $\Cc_i$ are numbered so that $\Cc_1$ is the principal class. By the orthogonality relations for additive characters mod $m$, it suffices to show  that for all nontrivial $m$th roots of unity $\zeta$,
\begin{equation}\label{eq:uniform} \sum_{\substack{(\alpha) \\ 0<|N(\alpha)| \le x}} \zeta^{\nu(\alpha)} = o(x), \qquad\text{as $x\to\infty$}.\end{equation}

Fix a nonzero, squarefull ideal $\ssf$ of $\Z_K$. We study the contribution to \eqref{eq:uniform} from $(\alpha)$ with squarefull part $\ssf$ in $\Id(\Z_K)$. Write $(\alpha)=\mathfrak{s}\mathfrak{u}$, so that $\mathfrak{u}$ is a squarefree ideal, $\ssf$ and $\uu$ are comaximal, and $[\uu] = [\ssf]^{-1}$. Then
\begin{equation}\label{eq:nualphaexpr} \nu(\alpha) = \sum_{\tau \in \mathcal{T}} \nu_{\tau}(\alpha) = \omega_1(\alpha) + \sum_{\substack{\tau \in \mathcal{T} \\ t_{1}(\tau)=0}} \nu_{\tau}(\alpha), \end{equation}
and
\[  \omega_1(\alpha) = \omega_1(\ssf) + \omega_1(\uu). \]
(We used here that $(1,0,\dots,0)$ is the only type in $\mathcal{T}$ with a nonvanishing $t_1$ coefficient, and that $\nu_{(1,0,\dots,0)}(\alpha)=\omega_1(\alpha)$.) For each $\tau\in \mathcal{T}$ with $t_{1}(\tau)= 0$, let
\[ \mathcal{D}(\ssf,\tau) = \{\dd \in \Id(\Z_K): \dd\mid \ssf,~\Omega_i(\dd) \le t_i(\tau)\text{ for all $i=1,2,\dots,h$}\}. \]
Any irreducible of type $\tau$ dividing $\alpha$ can be written uniquely as the product of an element of $\mathcal{D}(\ssf,\tau)$ and a cofactor relatively prime to $\ssf$. Thus,
\begin{equation}\label{eq:binomialexpression} \nu_{\tau}(\alpha) = \sum_{\dd \in \mathcal{D}(\ssf,\tau)} \prod_{i=1}^{h}\binom{\Omega_i(\alpha)-\Omega_i(\ssf)}{t_i(\tau) - \Omega_i(\dd)} = \sum_{\dd \in \mathcal{D}(\ssf,\tau)} \prod_{i=2}^{h}\binom{\omega_i(\uu)}{t_i(\tau) - \Omega_i(\dd)}. \end{equation}
Keeping in mind the universal bound $t_i(\tau) \le h$, it follows from \eqref{eq:binomialexpression} that the residue class class mod $m$ of
\[ \sum_{\substack{\tau \in \mathcal{T} \\ t_{1}(\tau)=0}} \nu_{\tau}(\alpha) \]
depends only on the vector
\[ (\omega_2(\uu)\bmod{mh!}, \dots, \omega_h(\uu) \bmod{mh!}) \in (\Z/mh!\Z)^{h-1}. \] Consequently, finite Fourier theory implies that
\[ \zeta^{\sum_{\substack{\tau \in \mathcal{T},~t_{1}(\tau)=0}} \nu_{\tau}(\alpha)} \]
can be written as a finite $\C$-linear combination of terms of the form
\[ \zeta_2^{\omega_2(\uu)} \cdots \zeta_{h}^{\omega_h(\uu)}, \]
where $\zeta_2, \dots, \zeta_{h}$ are $(mh!)$th roots of unity. Referring back to \eqref{eq:nualphaexpr}, we see that $\zeta^{\nu(\alpha)}$ is a finite $\C$-linear combination of expressions of the form
\[ \zeta^{\omega_1(\uu)} \zeta_2^{\omega_2(\uu)} \cdots \zeta_{h}^{\omega_h(\uu)}. \]
Thus,
\[ \sum_{\substack{(\alpha) \\ \text{squarefull part $\ssf$} \\ 0 <|N(\alpha)| \le x}} \zeta^{\nu(\alpha)} \]
is a finite $\C$-linear combination of sums of the form
\[ \sum_{\substack{N\uu \le x/N(\ssf) \\ [\uu] = [\ssf]^{-1}}} \mathbf{1}_{\gcd(\uu,\ssf)=1} \cdot \mu^2(\mathfrak{u}) \cdot \zeta^{\omega_1(\uu)} \zeta_2^{\omega_2(\uu)} \cdots \zeta_{h}^{\omega_h(\uu)}. \]
For all choices of $(mh)!$th roots of unity $\zeta_2,\dots,\zeta_h$, the function
\[ \mathfrak{u} \mapsto \mathbf{1}_{\gcd(\uu,\ssf)=1} \cdot \mu^2(\mathfrak{u}) \cdot \zeta^{\omega_1(\uu)} \zeta_2^{\omega_2(\uu)} \cdots \zeta_{h}^{\omega_h(\uu)}\]
satisfies the conditions of Lemma \ref{lem:idealclass}. Hence, the contribution to \eqref{eq:uniform} from $(\alpha)$ with a fixed squarefull part is $o(x)$, as $x\to\infty$.

We now finish the proof of Theorem \ref{thm:equidistribution}. Let $\epsilon > 0$. Let $B$ be a large, fixed real number. Continuing to use $\ssf$ for the squarefull part of $(\alpha)$, we see that the contribution to \eqref{eq:uniform} from $(\alpha)$ with $N\ssf \le B$ is $o(x)$, as $x \to\infty$. On the other hand,
\[ \Bigg|\sum_{\substack{(\alpha) \\ 0 < |N(\alpha)| \le x \\ N\ssf > B}} \zeta^{\nu(\alpha)}\Bigg| \le \sum_{\substack{\ssf\text{ squarefull} \\ B < N\ssf \le x}} \sum_{\substack{N\aa \le x \\ \ssf \mid \aa}} 1 \ll x\sum_{\substack{\ssf\text{ squarefull} \\ N\ssf >B}} \frac{1}{N\ssf}. \]
The sum appearing here is the tail of a convergent series, since
\[ \sum_{\ssf\text{ squarefull}}\frac{1}{N\ssf} = \prod_{\pp}\left(1 + \frac{1}{N\pp^2} + \frac{1}{N\pp^3} + \dots \right) <\infty.\]
So if we choose $B$ sufficiently large, those $(\alpha)$ with $N\ssf>B$ contribute less than $\epsilon x$. Since $\epsilon >0$ is arbitrary, Theorem \ref{thm:equidistribution} follows.
\end{proof}

\begin{remark} The distribution of $\delta(\alpha)$ in residue classes is much more complicated than that of $\nu(\alpha)$, even in the case $K=\Q$ (for which see \cite[Chapter 5, \S3]{narkiewicz86}).
\end{remark}

\section*{Acknowledgements}
The author is supported by NSF award DMS-1402268. He thanks Pete L. Clark for helpful discussions about arithmetic in $\Prin(\Z_K)$, and he thanks Carl Pomerance for suggesting the consideration of the behavior of the $\Prin(\Z_K)$ divisor function.

\providecommand{\bysame}{\leavevmode\hbox to3em{\hrulefill}\thinspace}
\providecommand{\MR}{\relax\ifhmode\unskip\space\fi MR }
% \MRhref is called by the amsart/book/proc definition of \MR.
\providecommand{\MRhref}[2]{%
  \href{http://www.ams.org/mathscinet-getitem?mr=#1}{#2}
}
\providecommand{\href}[2]{#2}


\begin{thebibliography}{B{\"O}RS05}

\bibitem[Add57]{addison57}
A.W.~Addison, \emph{A note on the compositeness of numbers}, Proc. Amer. Math. Soc. \textbf{8} (1957), 151--154.

\bibitem[B{\"O}RS05]{BORS05}
D.M.~Bradley, A.E.~{\"O}zl{\"u}k, R.A.~Rozario, and C.~Snyder, \emph{The distribution of the irreducibles in an algebraic number
              field}, J. Aust. Math. Soc. \textbf{79} (2005), 369--390.

\bibitem[EK40]{EK40}
P.~Erd\H{o}s and M.~Kac, \emph{The {G}aussian law of errors in the theory of
  additive number theoretic functions}, Amer. J. Math. \textbf{62} (1940),
  738--742.

\bibitem[Ell80]{elliott80}
P.D.T.A. Elliott, \emph{Probabilistic number theory {II}: Central limit
  theorems}, Grundlehren der Mathematischen Wissenschaften, vol. 240,
  Springer-Verlag, Berlin-New York, 1980.


\bibitem[GHK06]{GHK06}
A.~Geroldinger and F.~Halter-Koch, \emph{Non-unique factorizations: Algebraic,
  combinatorial and analytic theory}, Pure and Applied Mathematics, vol. 278,
  Chapman \& Hall/CRC, Boca Raton, FL, 2006.

\bibitem[GR09]{GR09}
A.~Geroldinger and I.Z.~Ruzsa, \emph{Combinatorial number theory and additive group theory}, Advanced Courses in Mathematics, CRM Barcelona, Birkh\"auser Verlag, Basel, 2009.

\bibitem[GS07]{GS07}
A.~Granville and K.~Soundararajan, \emph{Sieving and the {E}rd{\H o}s-{K}ac
  theorem}, Equidistribution in number theory, an introduction, NATO Sci. Ser.
  II Math. Phys. Chem., vol. 237, Springer, Dordrecht, 2007, pp.~15--27.


\bibitem[Hal68]{halasz68}
G.~Hal\'{a}sz, \emph{\"{U}ber die {M}ittelwerte multiplikativer zahlentheoretischer {F}unktionen}, Acta Math. Acad. Sci. Hungar. \textbf{19} (1968), 365--403.

% \bibitem[Hal95]{hall95}
% R.R.~Hall, \emph{A sharp inequality of {H}al\'asz type for the mean value of a              multiplicative arithmetic function},Mathematika \textbf{42} (1995), 144--157.

\bibitem[HW08]{HW08}
G.H. Hardy and E.M. Wright, \emph{An introduction to the theory of numbers},
  sixth ed., Oxford University Press, Oxford, 2008.

\bibitem[Kac41]{kac41}
M.~Kac, \emph{Note on the distribution of values of the arithmetic function $d(m)$}, Bull. Amer. Math. Soc. \textbf{47} (1941), 815--817.

\bibitem[KL08]{KL08}
W.~Kuo and Y.-R. Liu, \emph{The {E}rd{\H o}s-{K}ac theorem and its
  generalizations}, Anatomy of integers, CRM Proc. Lecture Notes, vol.~46,
  Amer. Math. Soc., Providence, RI, 2008, pp.~209--216.

\bibitem[Kro66]{dekroon66}
J.P.M.~de~Kroon, \emph{The asymptotic behaviour of additive functions in algebraic number theory}, Compos. Math. \textbf{17} (1965-1966), 207--261.

\bibitem[Lan99]{landau99}
E.~Landau, \emph{Neuer Beweis der Gleichung $\sum_{k=1}^{\infty}\frac{\mu(k)}{k}=0$}. Dissertation, Berlin, 1899.

\bibitem[Lan03]{landau03}
\bysame, \emph{\"{U}ber die zahlentheoretische Funktion $\mu(k)$}, SBer. Kais. Akad. Wissensch. Wien \textbf{112} (1903), 537--570.

\bibitem[Lan11]{landau11}
\bysame, \emph{\"{U}ber die \"{A}quivalenz zweier {H}aupts\"{a}tze der analytischen Zahlentheorie}, SBer. Kais. Akad. Wissensch. Wien \textbf{120} (1911), 973--988.

\bibitem[Lan18]{landau18}
\bysame, \emph{\"{U}ber {I}deale und {P}rimideale in {I}dealklassen}, Math.
  Z. \textbf{2} (1918), 52--154.

\bibitem[Liu04]{liu04}
Y.-R.~Liu, \emph{A generalization of the {E}rd\"os-{K}ac theorem and its
  applications}, Canad. Math. Bull. \textbf{47} (2004), 589--606.

\bibitem[LR01]{LR01}
L.~Lucht and K.~Reifenrath, \emph{Mean-value theorems in arithmetic semigroups}, Acta Math. Hungar. \textbf{93} (2001), 27--57.

\bibitem[Man97]{mangoldt97}
H.~von~Mangoldt, \emph{Beweis der Gleichung $\sum_{k=1}^{\infty}\frac{\mu(k)}{k} = 0$}, Sber. Kgl. Preu{\ss}. Akad. Wiss. Berlin (1897), 835--852.

\bibitem[Nar86]{narkiewicz86}
W.~Narkiewicz, \emph{Uniform distribution of sequences of integers in residue
              classes}, Lecture Notes in Mathematics, vol. 1087,
Springer-Verlag, Berlin, 1984.

\bibitem[Nar04]{narkiewicz04}
\bysame, \emph{Elementary and analytic theory of algebraic numbers},
  third ed., Springer Monographs in Mathematics, Springer-Verlag, Berlin, 2004.


\bibitem[Ols69]{olson69}
J.E.~Olson, \emph{A combinatorial problem on finite {A}belian groups. {I}}, J. Number Theory \textbf{1} (1969), 8--10.

\bibitem[Pil40]{pillai40}
S.S.~Pillai, \emph{Generalisation of a theorem of Mangoldt}, Proc. Indian Acad. Sci., Sect. A. \textbf{11} (1940), 13--20.

\bibitem[R{\'e}m66]{remond66}
P.~R{\'e}mond, \emph{\'{E}tude asymptotique de certaines partitions dans certains semi-groupes}, Ann. Sci. \'Ecole Norm. Sup. (3) \textbf{83} (1966), 343--410.

\bibitem[Sel39]{selberg39}
S.~Selberg, \emph{Zur {T}heorie der quadratfreien {Z}ahlen},
Math. Z. \textbf{44} (1939), 306--318.

\bibitem[Web96]{weber96}
H.~Weber, \emph{Ueber einen in der {Z}ahlentheorie angewandten {S}atz der
  {I}ntegralrechnung}, Nachr. Ges. Wiss. G\"ottingen, Math.-Phys. Kl. (1896),
  275--281.

\end{thebibliography}
\end{document}